\documentclass[12pt]{amsart}
\usepackage[margin=3cm, top=5cm]{geometry}
\usepackage{amsmath, amssymb, amsthm, pstricks}
\newtheorem{prop}{Proposition}[section]
\newtheorem{cor}[prop]{Corollary}
\newtheorem{thm}[prop]{Theorem}

\title{Alternating, pattern-avoiding permutations}
\author{Joel Brewster Lewis}

\begin{document}
\begin{abstract}
We study the problem of counting alternating permutations avoiding collections of permutation patterns including $132$.  We construct a bijection between the set $S_n(132)$ of $132$-avoiding permutations and the set $A_{2n + 1}(132)$ of alternating, $132$-avoiding permutations.  For every set $p_1, \ldots, p_k$ of patterns and certain related patterns $q_1, \ldots, q_k$, our bijection restricts to a bijection between $S_n(132, p_1, \ldots, p_k)$, the set of permutations avoiding $132$ and the $p_i$, and $A_{2n + 1}(132, q_1, \ldots, q_k)$, the set of alternating permutations avoiding $132$ and the $q_i$.  This reduces the enumeration of the latter set to that of the former.
\end{abstract}
\maketitle
\section{Introduction}
A classical problem asks for the number of permutations that avoid a certain permutation pattern.  This problem has received a great deal of attention (e.g., \cite{1342, 1234}) and has led to a number of interesting variations including the enumeration of permutations avoiding several patterns simultaneously \cite{le, simion} and the enumeration of special classes of pattern-avoiding permutations (e.g., involutions \cite{egge, simion} and derangements \cite{der}).  In this paper we consider a variation on this problem first studied by Mansour in \cite{mans}, namely the enumeration of \emph{alternating} permutations avoiding certain patterns.  Section \ref{back} provides definitions and some background.  In Section \ref{keybij} we construct a bijection $\varphi: S_n(132) \to A_{2n + 1}(132)$ that preserves pattern information and so allows us to relate many instances of our problem to the classical problem.  In particular, it provides a bijection between $S_n(132, p_1, \ldots, p_k)$ and $A_{2n + 1}(132, q_1, \ldots, q_k)$ for every set of patterns $p_1, \ldots, p_k$ and certain related patterns $q_1, \ldots, q_k$.  This reduces the enumeration of many classes of alternating, pattern-avoiding permutations to similar but simpler problems involving all permutations.  We give a few explicit examples, one of which (Corollary \ref{12k}) are already known by recursive, generating functional methods \cite{mans}; we believe the others to be original.

\section{Background and Definitions}\label{back}
A \emph{permutation} $w = w_1w_2\cdots w_n$ of length $n$ is a sequence containing each of the integers $1$ through $n$ exactly once.  A permutation is \emph{alternating} if $w_{1} < w_{2} > w_3 < w_4 > \ldots$.  (Note that in the terminology of \cite{EC1}, these ``up-down'' permutations are \emph{reverse alternating} while alternating permutations are ``down-up'' permutations, but that any result on either set can be translated into a result on the other via complementation, i.e. by considering $w'$ such that $w'_i = n + 1 - w_i$, so the pattern $132$ would be replaced by $312$ and so on.)  

Given a permutation $p \in S_k$ and a sequence $w = w_1\cdots w_n$, we say that $w$ \emph{contains the pattern $p$} if there exists a set of indices $1 \leq i_1 < i_2 < \ldots < i_k \leq n$ such that the subsequence $w_{i_1}w_{i_2}\cdots w_{i_k}$ of $w$ is order-isomorphic to $p$, i.e. the relative order of elements in the two sequences is the same.  Otherwise, $w$ is said to \emph{avoid} $p$.  Given patterns $p_1, \ldots, p_k$, we denote by $S_n(p_1, \ldots, p_k)$ the set of permutations of length $n$ avoiding all of the $p_i$ and by $A_n(p_1, \ldots, p_k)$ the set of alternating permutations of length $n$ avoiding all of the $p_i$, and we set $s_n(p_1, \ldots, p_k) = |S_n(p_1, \ldots, p_k)|$ and $a_n(p_1, \ldots, p_k) = |A_n(p_1, \ldots, p_k)|$.

The set $S_{n+1}(132)$ can be easily enumerated, as follows: if $w \in S_{n+1}(132)$, $w_{k+1} = n+1$ and $1 \leq i < k+1 < j \leq n+1$ then $w_i > w_j$ or else $w_i w_{k+1} w_j$ is a $132$-pattern contained in $w$.  Thus, $\{w_1, \ldots, w_{k}\} = \{n - k + 1, \ldots, n\}$, $\{w_{k + 2}, \ldots, w_{n+1}\} = \{1, \ldots, n - k\}$ and the sequences $w_1\cdots w_{k}$, $w_{k + 2}\cdots w_n$ are themselves $132$-avoiding.  Conversely, it's easy to check that if for some $k$ the permutation $w' = w'_1\cdots w'_{n+1}$ is such that $w'_{k+1} = n+1$, $\{w'_{k + 2}, \ldots, w'_{n+1}\} = \{1, \ldots, n - k\}$, $\{w'_{1}, \ldots, w'_{k}\} = \{n - k + 1, \ldots, n\}$ and the sequences $w'_1\cdots w'_{k}$, $w'_{k + 2}\cdots w'_{n+1}$ are $132$-avoiding then $w' \in S_{n+1}(132)$.  It follows immediately that
\[
s_{n+1}(132) = \sum_{k = 0}^n s_k(132)\cdot s_{n - k}(132)
\]
with initial value $s_0(132) = 1$, so $s_n(132)$ is equal to the $n$th Catalan number $C_n = \frac{1}{n + 1}\binom{2n}{n}$.  

A nearly identical argument can be applied to $a_{2n + 1}(132)$ -- one need only note that the position of the maximal element $2n + 3$ in a permutation $w \in A_{2n + 3}(132)$ must be even, so
\[
a_{2n + 3}(132) = \sum_{k = 0}^n a_{2k + 1}(132)\cdot a_{2n - 2k + 1}(132)
\]
with $a_1(132) = 1$ and thus $a_{2n + 1}(132) = C_n$ and $a_{2n + 1}(132) = s_n(132)$.  In the next section, we provide a bijective proof of this result.

\section{Key bijection}\label{keybij}

Given any permutation $w \in S_n$, we can bijectively associate a labeled, decreasing binary tree $T(w)$ as in Chapter 1 of \cite{EC1}: if $w_k$ is the maximal entry of $w$, we label the root with $w_k$, let the left subtree below the root be $T(w_1\cdots w_{k - 1})$ and let the right subtree be $T(w_{k + 1}\cdots w_n)$.  Note that for any $i \in [n]$, the vertex labeled $w_i$ has a left child in $T(w)$ if and only if $i > 1$ and $w_i > w_{i - 1}$ and has a right child if and only if $i < n$ and $w_i > w_{i + 1}$.  In particular, $T(w)$ is complete (i.e. every vertex has either zero or two children) if and only if $w$ is an alternating permutation of odd length.

If $w$ is $132$-avoiding then for any vertex $v$ in $T(w)$, each label on the left subtree at $v$ is larger than every label on the right subtree at $v$.  (This is essentially the same observation that we used to derive our recursion for $s_n(132)$ above.)  Erasing all labels gives a bijection between decreasing, labeled binary trees with this property and unlabeled binary trees and we will consider this bijection to be an identification, i.e. we will treat unlabeled trees and trees labeled in this way as interchangeable.

For any unlabeled binary tree $T$ we may construct the \emph{completion} $C(T)$ by adding vertices to $T$ so that every vertex of $T$ has two children in $C(T)$ and every vertex of $C(T)$ that is not also a vertex of $T$ has zero children in $C(T)$.  This operation is a bijection between unlabeled binary trees on $n$ nodes and unlabeled complete binary trees on $2n + 1$ nodes; to reverse it, erase all leaves.

We can use these three operations to construct a bijection between $S_n(132)$ and $A_{2n + 1}(132)$: beginning with any permutation $w \in S_n(132)$, construct the labeled tree $T(w)$.  Erase the labels from this tree and take its completion.  Take the associated labeled tree and apply the inverse of the bijection $T$.  The resulting permutation is an alternating, $132$-avoiding permutation of length $2n + 1$, i.e. an element of $A_{2n + 1}(132)$.  Each step is bijective, so the composition of steps is bijective.  We denote this bijection by $\varphi: S_n(132) \to A_{2n + 1}(132)$.  Figure \ref{applyphi} illustrates the application of $\varphi$ to the permutation $2341$.

\begin{figure}
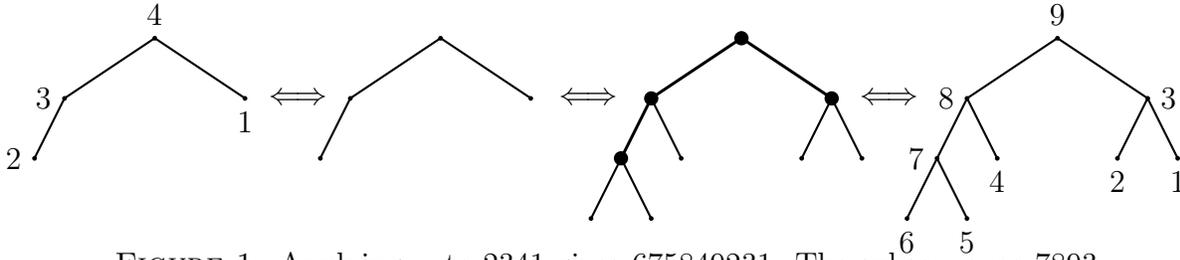

\begin{center}
\psset{unit=.4cm}
\pspicture(0, -2)(38,4)
\psdots*[dotscale=.5](0,0)(1,2)(4,4)(7,2)
\psline(0,0)(1,2)(4,4)(7,2)
\uput[u](4,4){$4$}
\uput[l](1,2){$3$}
\uput[l](0,0){$2$}
\uput[d](7,2){$1$}
\uput[r](7.35,2){$\Longleftrightarrow$}
\psdots*[dotscale=.5](9.5,0)(10.5,2)(13.5,4)(16.5,2)
\psline(9.5,0)(10.5,2)(13.5,4)(16.5,2)
\uput[r](17,2){$\Longleftrightarrow$}
\psdots*[dotscale=1.5](19.5,0)(20.5,2)(23.5,4)(26.5,2)
\psdots*[dotscale=.5](18.5,-2)(20.5,-2)(21.5,0)(25.5,0)(27.5,0)
\psline[linewidth=.1](19.5,0)(20.5,2)(23.5,4)(26.5,2)
\psline(18.5,-2)(19.5,0)(20.5,-2)
\psline(20.5,2)(21.5,0)
\psline(25.5,0)(26.5,2)(27.5,0)
\uput[0](27,2){$\Longleftrightarrow$}
\psdots*[dotscale=.5](30,0)(31,2)(34,4)(37,2)(29,-2)(31,-2)(32,0)(36,0)(38,0)
\psline(30,0)(31,2)(34,4)(37,2)
\psline(29,-2)(30,0)(31,-2)
\psline(31,2)(32,0)
\psline(36,0)(37,2)(38,0)
\uput[u](34,4){$9$}
\uput[l](31,2){$8$}
\uput[l](30,0){$7$}
\uput[d](29,-2){$6$}
\uput[d](31,-2){$5$}
\uput[d](32,0){$4$}
\uput[r](37,2){$3$}
\uput[d](36,0){$2$}
\uput[d](38,0){$1$}
\endpspicture
\caption{Applying $\varphi$ to $2341$ gives $675849231$.  The subsequence $7893$ of even-indexed entries is order-isomorphic to $2341$.}
\label{applyphi}
\end{center}
\end{figure}

Note that $\varphi^{-1}$ has a particularly nice form: erasing the leaves of $T(\varphi(w))$ is the same as erasing the odd-indexed entries in $\varphi(w)$, so $w$ is just the permutation order-isomorphic to the sequence $(\varphi(w)_2, \varphi(w)_4, \ldots)$ and $\varphi(w)$ is the unique member of $A_{2n + 1}(132)$ whose sequence of even-indexed entries is order-isomorphic to $w$.  It follows immediately that if $w$ contains a pattern $p$ then $\varphi(w)$ does as well.  In fact, we can say substantially more: 

\begin{prop}\label{main}
If a permutation $w \in S_n(132)$ contains the pattern $p \in S_k(132)$ then $\varphi(w)\in A_{2n + 1}(132)$ contains $\varphi(p)\in A_{2k + 1}(132)$.
\end{prop}
Note that the restriction $p \in S_k(132)$ is actually no restriction at all: pattern containment is transitive, so if $w$ avoids $132$ then $w$ automatically avoids all patterns that contain $132$.  
\begin{proof}
Suppose $w$ contains $p$.  Choose a sequence $1 \leq i_1 < i_2 < \ldots < i_k \leq n$ such that $w_{i_1}w_{i_2}\cdots w_{i_k}$ is an instance of $p$ contained in $w$.  Note that each vertex of $T(w)$ corresponds naturally via our construction to a vertex of $T(\varphi(w))$, and similarly for $T(p)$ and $T(\varphi(p))$.  For each $j$, denote by $v_{j}$ the vertex of $T(\varphi(w))$ corresponding to the vertex labeled $w_{i_j}$ in $T(w)$ and denote by $v'_j$ the vertex of $T(\varphi(p))$ corresponding to the vertex labeled $p_j$ in $T(p)$.  

We build a set of $2k + 1$ vertices of $T(\varphi(w))$ as follows: first, we include the vertices $v_{j}$ for $1 \leq j \leq k$.  Then, for each $j$, we add the left child of $v_{j}$ to our set if and only if $v'_j$ has a left child in $T(\varphi(p))$ but the vertex labeled $p_j$ in $T(p)$ does not -- that is, take the left child of $v_j$ if and only if a new left child was added to the vertex labeled $p_j$ when passing from $T(p)$ to $T(\varphi(p))$.  Do the same for right children.  Figure \ref{doublephi} illustrates this process for $p = 4312$ and $w = 5647231$.  

The relative order on the labels of vertices in this collection is determined entirely by the relative position of these vertices in the tree,\footnote{By the ``relative position'' of vertices $a$ and $b$, we mean the following: either one of $a$ and $b$ is a descendant of the other, in which case the ancestor has the larger label, or they have some nearest common ancestor $c$.  In the latter case, we say $a$ is ``relatively to the left'' of $b$ if $a$ lies in the left subtree of $c$ and $b$ lies in the right subtree of $c$.  By the $132$-avoidance of the permutations associated to these trees, ``being relatively left'' implies ``having a larger label.''} and the relative positions of these vertices in $T(\varphi(w))$ are the same as their relative positions in $T(\varphi(p))$.  Thus the labels on these vertices form a $\varphi(p)$-pattern contained in $\varphi(w)$.
\end{proof}

\begin{figure}
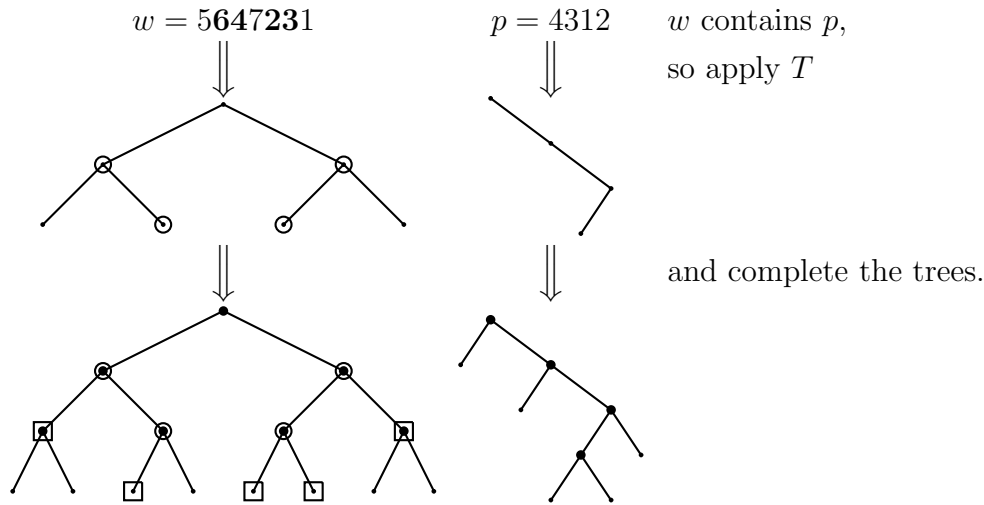

\begin{center}
\psset{unit=.4cm}
\begin{tabular}{ccl}
$w = 5{\bf 64}7{\bf 23}1$ & $p = 4312$ &$w$ contains $p$, \\
$\Big\Downarrow$ & $\Big\Downarrow$ & so apply $T$ \\
\pspicture(1,1.7)(13,6)
\psdots*[dotscale=.5](1,2)(3,4)(5,2)(7,6)(9,2)(11,4)(13,2)
\psline(1,2)(3,4)(5,2)
\psline(9,2)(11,4)(13,2)
\psline(3,4)(7,6)(11,4)
\pscircle(3,4){.3}
\pscircle(5,2){.3}
\pscircle(11,4){.3}
\pscircle(9,2){.3}
\endpspicture 
& 
\pspicture(2,1.5)(6,6)
\psdots*[dotscale=.5](2,6)(4,4.5)(6,3)(5,1.5)
\psline(2,6)(4,4.5)(6,3)(5,1.5)
\endpspicture
& \\
$\Big\Downarrow$ & $\Big\Downarrow$ & and complete the trees. \\
\pspicture(0,-.3)(14,6.3)
\psdots*[dotscale=1](1,2)(3,4)(5,2)(7,6)(9,2)(11,4)(13,2)
\psline(1,2)(3,4)(5,2)
\psline(9,2)(11,4)(13,2)
\psline(3,4)(7,6)(11,4)
\pscircle(3,4){.3}
\pscircle(5,2){.3}
\pscircle(11,4){.3}
\pscircle(9,2){.3}
\psdots*[dotscale=.5](0,0)(2,0)(4,0)(6,0)(8,0)(10,0)(12,0)(14,0)
\psline(0,0)(1,2)(2,0)
\psline(4,0)(5,2)(6,0)
\psline(8,0)(9,2)(10,0)
\psline(12,0)(13,2)(14,0)
\psline(1.3,2.3)(1.3,1.7)(.7,1.7)(.7,2.3)(1.3,2.3)
\psline(4.3,.3)(4.3,-.3)(3.7,-.3)(3.7,.3)(4.3,.3)
\psline(8.3,.3)(8.3,-.3)(7.7,-.3)(7.7,.3)(8.3,.3)
\psline(10.3,.3)(10.3,-.3)(9.7,-.3)(9.7,.3)(10.3,.3)
\psline(13.3,2.3)(13.3,1.7)(12.7,1.7)(12.7,2.3)(13.3,2.3)
\endpspicture
& 
\pspicture(1,0)(7,6)
\psdots*[dotscale=1](2,6)(4,4.5)(6,3)(5,1.5)
\psline(2,6)(4,4.5)(6,3)(5,1.5)
\psdots*[dotscale=.5](1,4.5)(3,3)(7,1.5)(4,0)(6,0)
\psline(1,4.5)(2,6)
\psline(3,3)(4,4.5)
\psline(7,1.5)(6,3)
\psline(4,0)(5,1.5)(6,0)
\endpspicture & 
\end{tabular}
\caption{The subsequence $6423$ of $w = 5647231$ is an instance of $p = 4312$.  We apply $T$ and complete both trees.  For each (circled) node of $C(T(w))$ corresponding to a node of $T(p)$ we select (box) a child if and only if the corresponding child was added to $T(p)$ in the completion $C(T(p))$.  The resulting subsequence $13\,14\,9\,10\,5\,6\,4\,7\,3$ is an instance of $\varphi(p) = 896734251$ in $\varphi(w) = 12\,13\,11\,14\,9\,10\,8\,15\,5\,6\,4\,7\,2\,3\,1$.}
\end{center}
\label{doublephi}
\end{figure}

The converse of this result is also true, and we prove a modest strengthening of it.
\begin{prop}\label{conv} 
Suppose $\varphi(w)\in A_{2n + 1}(132)$ contains a pattern $q$ and $q$ contains $p$, with the additional restriction that there is some subsequence of $q$ order-isomorphic to $p$ that contains no left-to-right minima of $q$.  Then $w \in S_n(132)$ contains $p$.
\end{prop}
\begin{proof}
If $\varphi(w)_{2k - 1} < \varphi(w)_{2k + 1}$ then $\varphi(w)_{2k - 1} < \varphi(w)_{2k + 1} < \varphi(w)_{2k}$ is a $132$-pattern contained in $\varphi(w)$, a contradiction, so $\varphi(w)_1 > \varphi(w)_3 > \varphi(w)_5 > \ldots$ and the left-to-right minima of $\varphi(w)$ are exactly the odd-indexed entries.  Select an instance of $q$ in $\varphi(w)$ and choose the special subsequence corresponding to $p$.  No entry of this subsequence is a left-to-right minimum in our instance of $q$, so no entry of this subsequence is a left-to-right minimum in $\varphi(w)$.  Then every entry of this subsequence (an instance of the pattern $p$) occurs among the even-indexed entries of $\varphi(w)$.  Since the sequence of even-indexed entries is order-isomorphic to $w$, $p$ is contained in $w$.
\end{proof}
The converse of Proposition \ref{main} follows because the even-indexed entries of $\varphi(p)$ are an instance of $p$ including no left-to-right minima in $\varphi(p)$, so if $\varphi(w)$ contains $\varphi(p)$ then $w$ contains $p$.

\section{Consequences}\label{consequences}
We can rephrase the results of the previous section as follows: 
\begin{thm}
For any collection of $132$-avoiding patterns $p_1, \ldots, p_k$, choose patterns $q_1, \ldots, q_k$ such that $q_i$ is contained in $\varphi(p_i)$ and there is an instance of $p_i$ contained in $q_i$ containing no left-to-right minima of $q_i$.  Then $\varphi$ is a bijection between $S_n(132, p_1, \ldots, p_k)$ and $A_{2n + 1}(132, q_1, \ldots, q_k)$, and in particular these sets are equinumerous.\end{thm}
Note that $\varphi(p_i)$ is always a valid choice for $q_i$, but that there may also be others.  Also, because $\varphi$ is a bijection on classes of pattern-avoiding permutations and on their complements, this restatement does not actually capture the full strength of our previous results.
\begin{proof}
We have by Proposition \ref{main} that the image of $S_n(132, p_1, \ldots, p_k)$ under $\varphi$ contains $A_{2n + 1}(132, \varphi(p_1), \ldots, \varphi(p_k))$ while we have by Proposition \ref{conv} that $\varphi$ is an injection from $S_n(132, p_1, \ldots, p_k)$ to $A_{2n + 1}(132, q_1, \ldots, q_k)$.  By transitivity of pattern containment, $A_{2n + 1}(132, \varphi(p_1), \ldots, \varphi(p_k)) \supseteq A_{2n + 1}(132, q_1, \ldots, q_k)$.  Thus in fact $\varphi$ is a bijection between $S_n(132, p_1, \ldots, p_k)$ and $A_{2n + 1}(132, q_1, \ldots, q_k)$, as claimed.\end{proof}  

As a result of this theorem, a large class of enumeration problems for pattern-avoiding alternating permutations can be expressed as enumeration problems for pattern-avoiding (standard) permutations.  We give three examples involving a single pattern.

\begin{cor}\label{12k}
For all $n \geq 0$, $k \geq 1$, $s_n(132, 12\cdots k) = a_{2n + 1}(132, 12\cdots (k + 1))$.  As a result, $w \in S_n(132)$ has longest increasing subsequence of length $k$ if and only if $\varphi(w) \in A_{2n + 1}(132)$ has longest increasing subsequence of length $k + 1$.
\end{cor}
\begin{proof}
We have that $\varphi(12\cdots k) = k(k + 1)(k -1)(k + 2)(k - 2)\cdots 2(2k + 1)1$ contains the subsequence $k(k + 1)(k + 2)\cdots (2k)(2k + 1)$.  This is an instance of the pattern $q = 12\cdots k(k + 1)$ and $q$ contains an instance of $12\cdots k$ not including any left-to-right minima of $q$, so $\varphi$ is a bijection between $S_n(132, 12\cdots k)$ and $A_{2n + 1}(132, 12\cdots k(k + 1))$, the first half of the claim.  

For the second half, note that the set of permutations in $S_n(132)$ with longest increasing subsequence of length $k$ is $S := S_n(132, 12\cdots k(k + 1))\setminus S_n(132, 12\cdots k)$.  Since $A_{2n + 1}(132, 12\cdots (k+ 1)) \subset A_{2n + 1}(132, 12\cdots (k + 1)(k + 2))$ and $S_n(132, 12\cdots k)\subset S_n(132, 12\cdots k(k + 1))$, $\varphi$ is a bijection between $S$ and $A_{2n + 1}(132, 12\cdots (k+ 1)(k + 2)) \setminus A_{2n + 1}(132, 12\cdots (k + 1))$, the set of $132$-avoiding alternating permutations with longest increasing subsequence of length $k + 1$.
\end{proof}
For $k = 1$ the resulting sequence is $a_{2n + 1}(132, 123) = s_n(132, 12) = 1$.  For $k = 2$ it's $a_{2n + 1}(132, 1234) = s_n(132, 123) = \left\lceil 2^{n - 1}\right\rceil$ (see \cite{simion}).  For $k = 3$ it's the even Fibonacci sequence $a_{2n + 1}(132, 12345) = s_n(132, 1234) = F_{2n - 2}$ defined with initial terms $F_0 = F_1 = 1$ (see \cite{west}), and so on.  These values were calculated directly for the alternating permutations in \cite{mans}, Section 2.

\begin{cor}\label{21k}
For all $n \geq 0$, $k \geq 2$, $a_{2n + 1}(132, 341256\cdots(k + 2)) = s_n(132, 2134\cdots k)$.
\end{cor}
\begin{proof}
We have $\varphi(2134\cdots k) = (k + 2)(k + 3)k(k + 1)(k - 1)(k + 4)(k - 2)(k + 5) \cdots 2(2k + 1)1$.  This contains the subsequence $(k + 2)(k + 3)k(k + 1)(k + 4)(k + 5) \cdots (2k)(2k + 1)$, an instance of the pattern $q = 341256\cdots (k + 2)$.  This $q$ contains the subsequence $(k + 3)(k + 1)(k + 4)(k + 5)\cdots (2k)(2k + 1)$, an instance of $2134\cdots k$ containing no left-to-right minima of $q$.  Then the result follows.
\end{proof}
The resulting sequences are the same as in the previous case (see \cite{simion, west}).  

Note that in both Corollaries \ref{12k} and \ref{21k} we chose the shortest possible pattern $q$ associated with $p = 123\cdots k$ or $p = 213\cdots k$, respectively.  In both cases there are many other possible choices for $q$ of various lengths: we could take any $q'$ that is order-isomorphic to a subsequence of $\varphi(p)$ that contains our selected instance of $q$.  In our last example, the choice of $q$ is much more restricted:
\begin{cor}\label{k21}
For all $n \geq 0$, $k \geq 2$, $a_{2n + 1}(132, (2k - 1)(2k)(2k - 3)(2k - 2)\cdots3412) = s_n(132, k(k - 1)\cdots 21)$.
\end{cor}
\begin{proof}
We have $\varphi(k(k - 1)\cdots 21) = (2k)(2k + 1)(2k - 2)(2k - 1)\cdots45231$, and taking $q$ to be the permutation order-isomorphic to $\varphi(p)$ with the final $1$ removed gives the result.
\end{proof}
For $k = 2$ the resulting sequence is $a_{2n + 1}(132, 3412) = s_n(132, 21) = 1$ and for $k = 3$ it's $a_{2n + 1}(132, 563412)= s_n(132, 321) = \frac{n^2 - n + 2}{2}$ (see \cite{simion}).  The $q$ selected in this instance is the only possible choice other than $q' = \varphi(p)$.

\section{A comment on even-length permutations}
Our preceding results all concern alternating permutations of odd length.  However, there is a simple relationship between $132$-avoiding alternating permutations of length $2n + 1$ and $132$-avoiding alternating permutations of length $2n$.  By the observation at the beginning of the proof of Proposition \ref{conv}, if $w \in A_{2n + 1}(132)$ then $w_{2n + 1} = 1$.  Thus, from any $132$-avoiding alternating permutation $w$ of length $2n + 1$ we can build a $132$-avoiding alternating permutation $w'$ of length $2n$ by removing the last entry of $w$ and subtracting $1$ from each other entry, i.e. $w' = (w_1 -1)(w_2 - 1)\cdots(w_{2n} - 1)$.  Conversely, given any $132$-avoiding alternating permutation $w'$ of length $2n$ we can build a $132$-avoiding alternating permutation $w$ of length $2n + 1$ by adding $1$ to each entry and appending a $1$, i.e. $w = (w'_1 + 1)(w'_2+1)\cdots(w'_{2n}+1)1$.  It follows immediately that $a_{2n}(132) = a_{2n + 1}(132)$.  Moreover, $w \in A_{2n + 1}(132)$ avoids a pattern $p = p_1\cdots p_{k - 1}p_k$ with $p_{k} = 1$ if and only if $w'$ avoids $p' = (p_1- 1)\cdots (p_{k-1} - 1)$, while if $p_k \neq 1$ then $w$ avoids $p$ if and only if $w'$ avoids $p$.  Then if $p_1, \ldots, p_i$ are patterns ending in $1$ and $q_1, \ldots, q_j$ are patterns not ending in $1$ we have $a_{2n + 1}(132, p_1, \ldots, p_i, q_1, \ldots, q_j) = a_{2n}(132, p_1', \ldots, p_i', q_1, \ldots, q_j)$, so in particular every enumeration problem concerning $132$-avoiding alternating permutations of even length reduces to a problem concerning permutations of odd length.

\bibliography{paperv4}{}

\begin{thebibliography}{1}

\bibitem{1342}
M.~B\'ona.
\newblock Exact enumeration of 1342-avoiding permutations; a close link with
  labeled trees and planar maps.
\newblock {\em J. Combinatorial Theory, Series A}, 80:257--272, 1997.

\bibitem{egge}
E.~S. Egge.
\newblock Restricted 3412-avoiding involutions, continued fractions, and
  {C}hebyshev polynomials.
\newblock {\em Advances in Applied Mathematics}, 33:451--475, 2004.

\bibitem{1234}
I.~M. Gessel.
\newblock Symmetric functions and {P}-recursiveness.
\newblock {\em J. Combinatorial Theory, Series A}, 53:257--285, 1990.

\bibitem{le}
I.~Le.
\newblock Wilf classes of pairs of permutations of length 4.
\newblock {\em Electronic J. Combinatorics}, 12:R25, 2005.

\bibitem{mans}
T.~Mansour.
\newblock Restricted $132$-alternating permutations and {C}hebyshev
  polynomials.
\newblock {\em Annals of Combinatorics}, 7:201--227, 2003.

\bibitem{der}
T.~Mansour and A.~Robertson.
\newblock Refined restricted permutations avoiding subsets of patterns of
  length three.
\newblock {\em Annals of Combinatorics}, 6:407--418, 2002.

\bibitem{simion}
R.~Simion and F.~W. Schmidt.
\newblock Restricted permutations.
\newblock {\em European J. Combinatorics}, 6:383--406, 1985.

\bibitem{EC1}
R.~P. Stanley.
\newblock {\em Enumerative Combinatorics, Volume 1}.
\newblock Cambridge University Press, 1997.

\bibitem{west}
J.~West.
\newblock Generating trees and forbidden subsequences.
\newblock {\em Discrete Mathematics}, 157:363--372, 1996.

\end{thebibliography}
\bibliographystyle{plain}










\end{document}